\documentclass[a4paper,reqno,12pt]{amsart}
\usepackage{verbatim}
\usepackage{amssymb,amsmath,amsthm}
\usepackage{amsfonts}
\usepackage{array}

\newtheorem{theorem}{Theorem}
\newtheorem{lemma}{Lemma}

\newtheorem{conjecture}{Conjecture}

\newtheorem{example}{Example}

\theoremstyle{remark}
\newtheorem{remark}[theorem]{\bf Remark}

\begin{document}

\setcounter{page}{1}

\title[]{Gaussian inequality}

\author{Tewodros~Amdeberhan  and David~Callan}

\address{Department of Mathematics,
Tulane University, New Orleans, LA 70118, USA}
\email{tamdeber@tulane.edu}

\address{Department of Statistics, 
University of Wisconsin-Madison, Madison, WI 53715, USA}
\email{callan@stat.wisc.edu}


\date{\today}

\begin{abstract} We prove some special cases of Bergeron's inequality involving two Gaussian polynomials (or $q$-binomials).
\end{abstract}

\maketitle

\section{Introduction}

\bigskip
\noindent
We begin by recalling the $q$-analogues $[n]!_q=\prod_{j=1}^n\frac{1-q^j}{1-q}$ of factorials and the $q$-analogue
$\binom{n}k_q=\frac{[n]!_q}{[k]!_q[n-k]!_q}$ of binomial coefficients. Adopt the convention $[0]!_q=1$. It is well-known that these rational functions $\binom{n}k_q$ are polynomials, in $q$, also called \emph{Gaussian polynomials}, having non-negative coefficients which are also \emph{unimodal} and symmetric. Furthermore, there are several combinatorial interpretations of which we state two of them.

\bigskip
\noindent
A \emph{word} of length $n$ over the \emph{alphabet} set $\{0,1\}$ is a finite sequence $w=a_1\cdots a_n$. Construct 
$$\mathcal{W}_{n,k}=\{w=a_1\cdots a_n: \text{$w$ has $k$ zeros and $n-k$ ones}\}$$
and the \emph{inversion set} of $w$ as $\text{Inv}(w)=\{(i,j): \text{$i<j$ and $a_i>a_j$}\}$.
The corresponding \emph{inversion number} of $w$ will be denoted $\text{inv}(w)=\#\text{Inv}(w)$. Then, we have
$$\binom{n}k_q=\sum_{w\in\mathcal{W}_{n,k}}q^{\text{inv}(w)}.$$

\noindent
Yet, another formulation which would come to appeal to many combinatorialists is
$$\binom{a+d}a_q=\sum_{T}q^{area(T)}$$
where $T$ is a lattice path inside an $a\times d$ box and $area(T)$ is area above the curve $T$. 

\bigskip
\noindent
Given two polynomials $f(q)$ and $g(q)$, we write $f(q)\geq g(q)$ provided that $f(q)-g(q)$ has non-negative coefficients in the powers of $q$. 

\smallskip
\noindent
The well-known \emph{Foulkes conjecture} (see, for instance \cite{B2}) was generalized by Vessenes \cite{V}. She conjectured that 
\begin{align}
(h_b\circ h_c) - (h_a \circ h_d)
\end{align}
is \emph{Schur positive} (expands with positive integer coefficients in the Schur basis $\{s_{\mu}\}_{\mu\vdash n}$ of symmetric polynomials) whenever $a\leq b<c \leq d$, with $n=ad=bc$, and one writes $(h_n\circ h_k)$ for the \emph{plethysm} of complete homogeneous symmetric functions. A well-known fact is that $(h_n\circ h_k)(1,q)=\binom{n+k}{k}_q$. Moreover, any non-zero evaluation of a Schur function at $1$ and $q$ is of the form $q^i+q^{i+1}+\dots + q^j$ for some $i<j$. Exploiting these facts on the occasion of \cite{B}, and assuming that (1) holds, F. Bergeron (see also [4]) underlined that the evaluation of the difference in (1), at $1$ and $q$, would imply the following: 

\noindent
\begin{conjecture} Assume $0<a\leq b<c\leq d$ are positive integers with $ad=bc$. Then, the following difference of two Gaussian polynomials is symmetric and satisfies
\begin{align} \binom{b+c}b_q-\binom{a+d}a_q\geq0. \end{align}
\end{conjecture}

\smallskip
\noindent
One can associate a direct combinatorial meaning to Vessenes' conjecture in the context of representation theory of $GL(V)$. Indeed, if it holds true, it would signify that there is an embedding of the composite of symmetric powers $S^a(S^d(V))$ inside $S^b(S^c(V)) $, as $GL(V)$-modules. It may however be more natural to state that there is a surjective $GL(V)$-module morphism the other way around (which is also equivalent). Therefore each $GL(V)$-irreducible occurs with smaller multiplicity in $S^a(S^d(V))$ than it does in $S^b(S^c(V))$, and the conjecture reflects this at the level of the corresponding characters (with Schur polynomials appearing as characters of irreducible representations). 

\bigskip
\noindent
The sole attempt \cite{Z} toward resolving Conjecture 1 was made by F. Zanello, who attends to the special case $a\leq 3$, including the property of symmetry and unimodality.  A sequence of numbers is \emph{unimodal} if it does not increase strictly after a strict decrease. The author in \cite{Z} offers a strengthening of Conjecture 1 to the effect that

\begin{conjecture} Preserve the hypothesis in Conjecture 1. Then, the coefficients of the symmetric polynomial
\begin{align*} \binom{b+c}b_q-\binom{a+d}a_q \end{align*}
are non-negative and unimodal.
\end{conjecture}

\noindent
Notice that symmetry is clear, since both $\binom{b+c}b_q$ and $\binom{a+d}a_q$ are symmetric polynomials of the same degree, $ad=bc$.
We started out this project with the goal of  proving the below $3$-parameter special case of Conjecture 1 which we dubbed the \emph{$\beta$-Conjecture}. Namely,

\begin{conjecture} For integers $0<a< b$ and $\beta\geq1$, we have
\begin{align}\binom{b+\beta a}b_q\geq \binom{a+\beta b}a_q. \end{align} \end{conjecture}

\smallskip
\noindent
The case $\beta=1$ is trivial. However, our journey in this effort failed short of capturing the $\beta$-Conjecture in its fullest. In the sequel, we supply the details of our success in settling the particular instance $\beta=2$. Let's commence by stating one useful identity.

\begin{theorem} ($q$-analogue Vandermonde-Chu). The following holds true
\begin{align}\sum_{j\geq0}\binom{X}{Z-j}_q\binom{Y}j_qq^{j(X-Z+j)}=\binom{X+Y}Z_q.\end{align} 
\end{theorem}
\begin{remark} In view of (4), Conjecture 1 tantamount
\begin{align*}
\binom{b+c}b_q=\sum_{k=0}^{\pmb{b}}\binom{b}k_q\binom{c}k_qq^{k^2} \geq \sum_{k=0}^{\pmb{a}}\binom{a}k_q\binom{d}k_qq^{k^2}=\binom{a+d}a_q. \end{align*}
\end{remark}

\section{THE CASE $\beta=2$ and $q=1$}

\noindent
In this section, we wish to explain the resolution of the $\beta$-Conjecture for the ordinary binomial coefficients ($q=1$) while $\beta=2$, which elaborates a natural development. 

\smallskip
\noindent
For $a<b<c<d$ with $ad=bc$ and special case $c=2a,d=2b$, say $b=a+i,\,i\ge 1$, it would be desirable to find a bijective proof for
$$\binom{3a+i}{a+i}\geq \binom{3a+2i}{a}\, .$$
An injection from a set counted by the smaller number to one counted by the larger number would be nice but a better proof would be an expression for the difference as a sum of obviously positive terms. For $i=1$, we have 
$$\binom{3a+1}{a+1}-\binom{3a+2}{a} = \binom{3a+1}{a-1}\, ,$$
and the right-hand side is clearly positive. It seems for general $i=1,2,\dots,$ 
$$\binom{3a+i}{a+i} - \binom{3a+2i}{a} = \sum_{k=1}^i c_k(i) \binom{3a+i}{a-k}$$
for integers $c_k(i)$ and, furthermore, the $c_k(i)$ \emph{are all positive}. Here is a table for $c_k(i)$ when $1\leq k\leq i\leq8$:
\[
\begin{array}{cccccccc}
 1 & \text{} & \text{} & \text{} & \text{} & \text{} & \text{} & \text{} \\
 3 & 1 & \text{} & \text{} & \text{} & \text{} & \text{} & \text{} \\
 6 & 6 & 1 & \text{} & \text{} & \text{} & \text{} & \text{} \\
 10 & 19 & 9 & 1 & \text{} & \text{} & \text{} & \text{} \\
 15 & 45 & 39 & 12 & 1 & \text{} & \text{} & \text{} \\
 21 & 90 & 120 & 66 & 15 & 1 & \text{} & \text{} \\
 28 & 161 & 301 & 250 & 100 & 18 & 1 & \text{} \\
 36 & 266 & 658 & 755 & 450 & 141 & 21 & 1 \\
\end{array}
\]
but appeared hard to get a handle on them. The evolution of our next progress begins with the discovery of
$$c_k(i)=\binom{i+k-1}{2k}+2\binom{i+k-1}{2k-1}-\binom{i}k.$$
Let's contract these coefficients as $c_k(i)=\frac{i+3k}{i+k}\binom{i+k}{2k}-\binom{i}k$, for $i, k\geq1$. Notice $c_0(i)=0$. We need some preliminary results.

\begin{lemma} We have
$$\binom{3a+2i}a=\sum_{k\geq0}\binom{i}k\binom{3a+i}{a-k}.$$ \end{lemma}
\begin{proof} This follows from the Vandermonde-Chu identity (Theorem 1 for $q=1$)
$$\binom{X+Y}Z=\sum_{k\geq0}\binom{X}k\binom{Y}{Z-k}$$
applied to $\binom{3a+2i}a=\binom{i+3a+i}a$ with $X=i, Y=3a+i$ and $Z=a$.
\end{proof}

\begin{lemma} We have 
$$\binom{3a+i}{a+i}=\sum_{k\geq0} \frac{i+3k}{i+k}\binom{i+k}{2k}\binom{3a+i}{a-k}
=\sum_{k\geq0} \left[\binom{i+k}{2k}+\binom{i+k-1}{2k-1}\right]\binom{3a+i}{a-k}.$$ \end{lemma}
\begin{proof} We implement Zeilberger's algorithm (from the Wilf-Zeilberger theory). Define 
$$F(i,k)=\frac{i+3k}{i+k}\cdot \frac{\binom{i+k}{2k}\binom{3a+i}{a-k}}{\binom{3a+i}{a+i}} \qquad \text{and} \qquad
G(i,k)=-\frac{\binom{i+k-1}{2k-2}\binom{3a+i}{a-k}}{\binom{3a+i}{a+i}}.$$
Check that $F(i+1,k)-F(i,k)=G(i,k+1)-G(i,k)$ and sum both sides over all integer values $k$. Then, notice the right-hand side vanishes and hence we obtain a sum $\sum_k F(i,k)$ that is \it constant \rm in the variable $i$. Determine this constant by substituting, say $i=1$,
$$\sum_{k=0}^1F(1,k)=\frac{\binom{3a+1}a}{\binom{3a+1}{a+1}}+\frac{2\binom{3a+1}{a-1}}{\binom{3a+1}{a+1}}=\frac{a+1}{2a+1}+\frac{a}{2a+1}=1.$$ 
Therefore, $\sum_kF(i,k)=1$, identically, for all $i\geq1$. The proof follows.
\end{proof}

\noindent
We now state the main result of this section.

\begin{theorem} We have
$$\binom{3a+i}{a+i}-\binom{3a+2i}a
=\sum_{k\geq1}\left\{\frac{i+3k}{i+k}\binom{i+k}{2k}-\binom{i}k \right\}\binom{3a+i}{a-k}.$$ \end{theorem}
\begin{proof} Immediate from Lemma 1 and Lemma 2.
\end{proof}

\begin{lemma} For $k\geq1$, the coefficients $c_k(i)$ are non-negative. \end{lemma}
\begin{proof} We may look at it in two different ways: 

\noindent
(1) $c_k(i)=\frac{2k}{i+k}\binom{i+k}{2k}+\binom{i+k}{2k}-\binom{i}k=\frac{2k}{i+k}\binom{i+k}{2k}+\binom{i+k}{i-k}-\binom{i}{i-k}$. Obviously  $\binom{i+k}{i-k}\geq\binom{i}{i-k}$,
 therefore $c_k(i)\geq0$.

\noindent
(2) $c_k(i)=\binom{i+k}{2k}+\binom{i+k-1}{2k-1}-\binom{i}k=\binom{i+k}{2k}+\sum_{r=0}^{k-2}\binom{i+r}{k+1+r}$ shows clearly that $c_k(i)\geq0$. The identity $\binom{i+k-1}{2k-1}=\binom{i}k+\sum_{r=1}^{k-1}\binom{i+r-1}{k+r}$ results from a cascading effect of the familiar binomial recurrence $\binom{u}v+\binom{u}{v-1}=\binom{u+1}{v}$.
\end{proof}

\section{ $q$-ANALOGUES when $\beta=2$}

\noindent
In the present section, we aim to generalize our proofs given in the preceding section by lifting the argument from the ordinary binomials to Gaussian polynomials.

\begin{lemma} We have
$$\binom{3a+2i}a_q=\sum_{k\geq0} q^{(a-k)(i-k)}\binom{i}k_q\binom{3a+i}{a-k}_q.$$ \end{lemma}
\begin{proof} This follows from the \it Vandermonde-Chu \rm identity (Theorem 1)
$$\binom{X+Y}Z_q=\sum_{k\geq0}q^{(Z-k)(X-k)}\binom{X}k_q\binom{Y}{Z-k}_q$$
on $\binom{3a+2i}a_q=\binom{i+3a+i}a_q$ with $X=i, Y=3a+i$ and $Z=a$.
\end{proof}

\begin{lemma} We have
\begin{align*}
\binom{3a+i}{a+i}_q
&=\sum_{k\geq0}q^{(a-k)(i-k)}\left[\binom{i+k}{2k}_q+q^{a+i}\binom{i+k-1}{2k-1}_q\right]\binom{3a+i}{a-k}_q.
\end{align*} \end{lemma}
\begin{proof} Let's rewrite $\binom{i+k}{2k}_q+q^{a+i}\binom{i+k-1}{2k-1}_q=\left[1+\frac{q^{a+i}(1-q^{2k})}{1-q^{i+k}}\right]\binom{i+k}{2k}_q$ and define the functions
\begin{align*} 
F(i,k)&=q^{(a-k)(i-k)}\left[1+q^{a+i}\cdot\frac{1-q^{2k}}{1-q^{i+k}}\right]\frac{\binom{i+k}{2k}_q\binom{3a+i}{a-k}_q}{\binom{3a+i}{a+i}_q}, \qquad \text{and} \\
G(i,k)&=-q^{(a-k+1)(i-k+1)}\cdot\frac{\binom{i+k-1}{2k-2}_q\binom{3a+i}{a-k}_q}{\binom{3a+i}{a+i}_q}.
\end{align*}
Divide both sides of the intended identity by $\binom{3a+i}{a+i}_q$. Our goal is to prove $\sum_kF(i,k)=1$ by adopting the Wilf-Zeilberger technique. To this end, calculate the following two ratios 
$$A(i,k):=\frac{F(i+1,k)}{F(i,k)}-1 \qquad \text{and} \qquad B(i,k):=\frac{G(i,k+1)}{F(i,k)}-\frac{G(i,k)}{F(i,k)}$$ 
resulting in
\begin{align*}
A(i,k)&=\frac{q^{a-k}(1-q^{i+k})(1-q^{a+i+1})(1-q^{i+k+1}+q^{a+i+1}-q^{a+i+2k+1})}{(1-q^{i-k+1})(1-q^{2a+i+k+1})(1-q^{i+k}+q^{a+i}-q^{a+i+2k})}-1 \qquad \text{and} \\
B(i,k)&=\left[-\frac{1-q^{a-k}}{1-q^{2a+ i+ k+1}}+\frac{q^{a+i-2k+1}(1-q^{2k})(1-q^{2k-1})}{(1-q^{i+k})(1-q^{i-k+1})}\right]\cdot
\frac{1-q^{i+k}}{1-q^{i+k}+q^{a+i}-q^{a+i+2k}}.
\end{align*}
Verify routinely $A(i,k)=B(i,k)$. Thus $F(i+1,k)-F(i,k)=G(i,k+1)-G(i,k)$. Now, sum both sides over all integer values $k$. Then, notice that the right-hand side vanishes and hence we obtain a sum $\sum_k F(i,k)$ that is \it constant \rm in the variable $i$. Determine this constant by substituting, say $i=1$ and proceed with some simplifications leading to 
$$\sum_{k=0}^1F(1,k)=q^a\cdot \frac{1-q^{a+1}}{1-q^{2a+1}}+\frac{1-q^a}{1-q^{2a+1}}=1.$$ 
Therefore, $\sum_kF(i,k)=1$, identically, for all $i\geq1$. The assertion follows.
\end{proof}

\begin{lemma} We have the identity
\begin{align*}
\binom{i+k}{2k}_q=\binom{i}k_q+\sum_{r=1}^kq^{k+r}\binom{i+r-1}{k+r}_q.
\end{align*}
\end{lemma}
\begin{proof} Use the recursive relations $\binom{a}b_q=\binom{a-1}b_q+q^{a-b}\binom{a-1}{b-1}_q=q^b\binom{a-1}b_q+\binom{a-1}{b-1}_q$.
\end{proof}

\begin{lemma} We have the inequality $\binom{i+k}{i-k}_q\geq\binom{i}{i-k}_q$.
\end{lemma}
\begin{proof} We use the interpretation of the Gaussian polynomials as the \emph{inversion number} generating function for all bit
strings of length $n$ with $k$ zeroes and $n-k$ ones, that is 
$$\binom{n}k_q=\sum_{w\in\, 0^k1^{n-k}}q^{\text{inv}(w)}.$$
Let $w'\in\, 0^{i-k}1^k\sqcup 1^k$ denote a bit where the last $k$ digits are all ones. In this sense, we get
\begin{align*}\binom{i+k}{i-k}_q=\sum_{w\in\,0^{i-k}1^{2k}}q^{\text{inv}(w)}
&=\sum_{w'\in\,0^{i-k}1^k\sqcup 1^k}q^{\text{inv}(w')}+\sum_{w'\not\in\,0^{i-k}1^k\sqcup 1^k}q^{\text{inv}(w')} \\
&=\sum_{w\in\,0^{i-k}1^k}q^{\text{inv}(w)}+\sum_{w'\not\in\,0^{i-k}1^k\sqcup 1^k}q^{\text{inv}(w')} \\
&=\binom{i}{i-k}_q+\sum_{w'\not\in\,0^{i-k}1^k\sqcup 1^k}q^{\text{inv}(w')}
\end{align*}
where we note that $\text{inv}(w')=\text{inv}(w)$ if the word $w'\in\,0^{i-k}1^k\sqcup 1^k$ is associated with $w\in\,0^{i-k}1^k$ found by dropping the last $k$ ones.
The assertion is now immediate.
\end{proof}

\noindent
We prove the main result of this section and our paper, the $\beta$-Conjecture for $\beta=2$.

\begin{theorem} The polynomial $P(q):=\binom{3a+i}{a+i}_q-\binom{3a+2i}a_q$ has non-negative coefficients. \end{theorem}
\begin{proof} From Lemma 4 and Lemma 5, we infer
\begin{align*}
P(q)=\sum_{k\geq1}q^{(a-k)(i-k)}\left[\binom{i+k}{2k}_q+q^{a+i}\binom{i+k-1}{2k-1}-\binom{i}k_q\right]\binom{3a+i}{a-k}_q.
\end{align*}
It suffices to verify positivity of the terms inside the inner parenthesis on the right-hand side. We may pair up these terms and compliment the lower index to the effect that
\begin{align*}
\binom{i+k}{2k}_q-\binom{i}k_q+q^{a+i}\binom{i+k-1}{2k-1}_q
=\binom{i+k}{i-k}_q-\binom{i}{i-k}_q+q^{a+i}\binom{i+k-1}{2k-1}_q.
\end{align*}
To reach the conclusion, simply apply Lemma 6 or Lemma 7.
\end{proof}

\section{Final remarks}

\noindent
In the present section, we close our discussion with one conjecture as a codicil of certain calculations we encountered while digging up ways to prove the $\beta$-Conjecture.

\begin{conjecture} For each $0\leq k\leq a<b$, we have
$$\binom{a}k_q\binom{a+b}{b-k}_q\geq \binom{b}k_q\binom{a+b}{a-k}_q \qquad \text{or} \qquad 
\binom{a}k_q\binom{b}k_q\binom{b+a}b_q\left[\frac1{\binom{a+k}k_q}-\frac1{\binom{b+k}k_q}\right]
\geq0.$$
\end{conjecture}

\noindent
The next elementary result might be helpful if one decides to engage this conjecture.

\begin{lemma} For $0\leq k\leq a<b$, we have
\begin{align*}
\frac1{\binom{a+k}k_q}-\frac1{\binom{b+k}k_q}
&=\sum_{i=1}^kq^{a+i}\frac{1-q^{b-a}}{1-q^{b+i}}\prod_{j=i}^k\frac{1-q^j}{1-q^{a+j}}\prod_{j=1}^{i-1}\frac{1-q^j}{1-q^{b+j}}.
\end{align*}
\end{lemma}
\begin{proof} This results from partial fractions. \end{proof}

\begin{example}
\begin{align*}
\frac1{\binom{a+1}1_q}-\frac1{\binom{b+1}1_q}
&=\frac{q^{a+1}(1-q)(1-q^{b-a})}{(1-q^{a+1})(1-q^{b+1})} \\
\frac1{\binom{a+2}2_q}-\frac1{\binom{b+2}2_q}
&=\frac{q^{a+1}(1-q)(1-q^2)(1-q^{b-a})}{(1-q^{a+1})(1-q^{a+2})(1-q^{b+1})}+\frac{q^{a+2}(1-q)(1-q^2)(1-q^{b-a})}{(1-q^{a+2})(1-q^{b+1})(1-q^{b+2})}.
\end{align*}
\end{example}

\begin{example}
\begin{align*}
\frac1{\binom{a+3}3_q}-\frac1{\binom{b+3}3_q}
&=\frac{q^{a+1}(1-q)(1-q^2)(1-q^3)(1-q^{b-a})}{(1-q^{a+1})(1-q^{a+2})(1-q^{a+3})(1-q^{b+1})} \\
&+\frac{q^{a+2}(1-q)(1-q^2)(1-q^3)(1-q^{b-a})}{(1-q^{a+2})(1-q^{a+3})(1-q^{b+1})(1-q^{b+2})} \\
&+\frac{q^{a+3}(1-q)(1-q^2)(1-q^3)(1-q^{b-a})}{(1-q^{a+3})(1-q^{b+1})(1-q^{b+2})(1-q^{b+3})}. \\
\end{align*}
\end{example}

\noindent
\begin{remark}
As a side note, we recall that G. E. Andrews \cite{GA} expresses $\binom{n}k_q-\binom{n}{k-1}_q$ as the generating function for partitions with particular \emph{Frobenius symbols}, while L. M. Butler \cite{LB} does this with the help of the \emph{Kostka-Foulkes polynomials} to show non-negativity of the coefficients. We shall provide an alternative algebraic approach.
\end{remark}

\begin{lemma} For $0\leq 2k\leq n$, we have $\binom{n}k_q-\binom{n}{k-1}_q\geq0$.
\end{lemma}
\begin{proof} Let $n=\alpha k+d$ where $0\leq d<k$. Rewrite
\begin{align*}
\binom{n}k_q-\binom{n}{k-1}_q
&=q^k\binom{n}{k-1}_q\frac{1-q^{(\alpha-2)k}}{1-q^k}
+q^{(\alpha-1)k}\binom{n}{k-1}_q\frac{1-q^{d+1}}{1-q^k}.
\end{align*}
Observe $\frac{1-q^{(\alpha-2)k}}{1-q^k}$ is already a polynomial with non-negative coefficients. Furthermore, since $U(q):=\binom{n}{k-1}_q$ is unimodal \cite{A}, \cite{HJ}, \cite{S}, 
the coefficient of $q^j$ in $U(q)\cdot(1-q^{d+1})$ is non-negative as long as $2j\leq\deg(U)$. The same is true for $U(q)\frac{1-q^{d+1}}{1-q^k}$ as a formal power series. 
Since the \it polynomial \rm $U(q)\frac{1-q^{d+1}}{1-q^k}$ is \it symmetric, \rm having degree no greater than $U(q)$,
all remaining coefficients of $U(q)\frac{1-q^{d+1}}{1-q^k}$ are non-negative. \end{proof}

\section*{Acknowledgement}

\noindent
The first author thanks R. P. Stanley for bringing Bergeron's conjecture to his attention, also F. Bergeron for his generous explanation of the problem studied in the present work.

\medskip

\end{document}